\documentclass[12pt]{amsart}
\usepackage{bm}
 \usepackage{graphicx}

 \newtheorem{thm}{Theorem}[section]
 \newtheorem{cor}[thm]{Corollary}
 
 \newtheorem{prop}[thm]{Proposition}
 \theoremstyle{definition}
 
 \theoremstyle{remark}
 \newtheorem{rem}[thm]{Remark}
 \numberwithin{equation}{section}

 \newcommand{\Real}{\mathbb{R}}

 \newcommand{\ds}{\mathfrak f}
 \newcommand{\Tan}{^*\nabla}
 \newcommand{\Rm}{\textbf{Rm}}
\begin{document}

\title[Comparison Theorems on Contact Subriemannian Manifolds]
{Bishop and Laplacian Comparison Theorems on Three Dimensional Contact Subriemannian Manifolds with Symmetry}

\author{Andrei Agrachev}
\email{agrachev@sissa.it}
\address{International School for Advanced Studies, via Bonomea 265, 34136, Trieste, Italy and Steklov Mathematical Institute, ul. Gubkina 8, Moscow, 119991 Russia}

\author{Paul W.Y. Lee}
\email{wylee@math.cuhk.edu.hk}
\address{Room 216, Lady Shaw Building, The Chinese University of Hong Kong, Shatin, Hong Kong}

\date{\today}
\thanks{The first author was partially supported by the PRIN project and the second author was supported by the NSERC postdoctoral fellowship.}

\begin{abstract}

We prove a Bishop volume comparison theorem and a Laplacian comparison theorem for three dimensional contact subriemannian manifolds with symmetry.
\end{abstract}

\maketitle

\section{Introduction}

Recently, there are numerous progress in the understanding of curvature type invariants in subriemannian geometry and their applications to PDE \cite{Ju,LiZe1,LiZe2,AgLe2,BaGa1,BaGa2,BaBoGa,LiZe3}. In this paper, we continue to investigate some consequences on bounds of these curvature invariants.  More precisely, we prove a Bishop comparison theorem and a Laplacian comparison theorem for three dimensional contact subriemannian manifolds with symmetry (also called Sasakian manifolds). Weaker results of volume comparison on Sasakian manifolds have previously been obtained in \cite{ChYa1}. We would like to thank Professor Chanillo for pointing this out.

The paper is organized as follows. In section \ref{subriemannian}, we recall various notions in subriemannian geometry needed in this paper. In particular, we recall the definition of curvature $R_{11}$ and $R_{22}$ for three dimensional contact subriemannian manifolds introduced in \cite{LiZe1,LiZe2,AgLe2}. In section \ref{Tanaka}, we show that the curvature $R_{11}$ is closely related to the Tanaka-Webster curvature in CR geometry. In section \ref{spaceform}, we collect various results on the cut loci of Sasakian manifold with constant Tanaka-Webster curvature (also called Sasakian space forms). In section \ref{volestimate}, we give an estimate for the volume of subriemannian balls. In section \ref{subBishop}, we prove the subriemannian Bishop theorem which compares the volume of subriemannian balls of a Sasakian manifold and a Sasakian space form. We introduce the subriemannian Hessian and sub-Laplacian in section \ref{HessLaplace} and give the formula for the Laplacian of the subriemannian distance in Sasakian space form in section \ref{LaplaceSpaceForm}. We prove a subriemannian Hessian and a subriemannian Laplacian comparison theorem in section \ref{subLaplace}. As an application, we give a lower bound of the solution to the subriemannian heat equation in section \ref{ChYau}.

\smallskip

\section*{Acknowledgment}
The authors would like to thank N. Garofalo for stimulating discussions.

\smallskip

\section{Subriemannian Geometry}\label{subriemannian}

In this section, we recall various notions in subriemannian geometry needed in this paper.
A subriemannian manifold is a triple $(M,\Delta,g)$, where $M$ is a smooth manifold, $\Delta$ is a distribution (a vector subbundle of the tangent bundle of $M$), and $g$ is a fibrewise inner product defined on the distribution $\Delta$. The inner product $g$ is also called a subriemannian metric. An absolutely continuous curve $\gamma:[0,1]\to M$ on the manifold $M$ is called horizontal if it is almost everywhere tangent to the distribution $\Delta$. We can use the inner product $g$ to define the length $l(\gamma)$ of a horizontal curve $\gamma$ by
\[
l(\gamma)=\int_0^1g(\dot\gamma(t),\dot\gamma(t))^{1/2}dt.
\]

The distribution $\Delta$ is called bracket-generating if vector fields contained in $\Delta$ together with their iterated Lie brackets span the whole tangent bundle. More precisely, let $\Delta_1$ and $\Delta_2$ be two distributions on a manifold $M$, and let $\mathfrak X(\Delta_i)$ be the space of all vector fields contained in the distribution $\Delta_i$. Let $[\Delta_1,\Delta_2]$ be the distribution defined by
\[
[\Delta_1,\Delta_2]_x=\text{span}\{[w_1,w_2](x)|w_i\in\mathfrak X(\Delta_i)\}.
\]

We define inductively the following distributions: $[\Delta,\Delta]=\Delta^2$ and $\Delta^k=[\Delta,\Delta^{k-1}]$. A
distribution $\Delta$ is called bracket generating if $\Delta^k=TM$ for some $k$. Under the bracket generating assumption, we have the following famous Chow-Rashevskii Theorem (see \cite{Mo} for a proof):

\begin{thm}(Chow-Rashevskii)
 Assume that the manifold $M$ is connected and the distribution $\Delta$ is bracket generating, then there is a horizontal curve joining any two given points.
\end{thm}

Assuming the distribution $\Delta$ is bracket generating, we can define the subriemannian or Carnot-Caratheodory distance $d(x,y)$ between two points $x$ and $y$ on the manifold $M$ is defined by
\begin{equation}\label{CCdistance}
 d(x,y)=\inf l(\gamma),
\end{equation}
where the infimum is taken over all horizontal curves which start from $x$ and end at $y$.

The horizontal curves which realize the infimum in (\ref{CCdistance}) are called length minimizing geodesics. From now on all manifolds are assumed to be a complete metric space with respect to a given subriemannian distance. In particular, a version of Hopf-Rinow theorem for subriemannian manifolds (\cite{BeRi}) guarantees that there is at least one geodesic joining any two given points.

Next we discuss the geodesic equation in the subriemannian setting. Let $\alpha$ be a covector in the cotangent space $T^*_xM$ at the point $x$. By nondegeneracy of the metric $g$, we can define a vector $v$ in the distribution $\Delta_x$ such that $g(v,\cdot)$ coincides with $\alpha(\cdot)$ on $\Delta_x$. The subriemannian Hamiltonian $H$ corresponding to the subriemannian metric $g$ is defined by
\[
 H(\alpha):=\frac{1}{2}g(v,v).
\]
Note that this construction defines the usual kinetic energy Hamiltonian in the Riemannian case.

Let $\pi:T^*M\to M$ be the projection map. The tautological one form $\theta$ on $T^*M$ is defined by
\[
 \theta_\alpha(V)=\alpha(d\pi(V)),
\]
where $\alpha$ is in the cotangent bundle $T^*M$ and $V$ is a tangent vector on the manifold $T^*M$ at $\alpha$.

Let $\omega=d\theta$ be the symplectic two form on $T^*M$. The Hamiltonian vector field $\vec H$ corresponding to the Hamiltonian $H$ is defined by $\omega(\vec H,\cdot)=-dH(\cdot)$. By the non-degeneracy of the symplectic form $\omega$, the Hamiltonian vector field  $\vec H$ is uniquely defined. We denote the flow corresponding to the vector field $\vec H$ by $e^{t\vec H}$. If $t\mapsto e^{t\vec H}(\alpha)$ is a trajectory of the above Hamiltonian flow, then its projection $t\mapsto \gamma(t)=\pi(e^{t\vec H}(\alpha))$ is a locally minimizing geodesic. This means that sufficiently short segment of the curve $\gamma$ is a minimizing geodesic between its endpoints. The minimizing geodesics obtained this way are called normal geodesics. In the special case where the distribution $\Delta$ is the whole tangent bundle $TM$, the distance function (\ref{CCdistance}) is the usual Riemannian distance and all geodesics are normal. However, this is not the case for subriemannian manifolds in general (see \cite{Mo} and reference therein for more detail).

Next we restrict our attention to the three dimensional contact subriemannian manifold. Let $\Delta$ be a bracket generating distribution with two dimensional fibres on a three dimensional manifold $M$. $\Delta$ is a contact distribution if there exists a covector $\sigma$ such that $\Delta=\{v|\sigma(v)=0\}$ and the restriction of $d\sigma$ to $\Delta$ is non-degenerate. If we fix a subriemannian metric $g$, then we can choose $\sigma$ so that the restriction of $d\sigma$ to the distribution $\Delta$ coincides with the volume form with respect to the subriemannian metric $g$.

Let $\{v_1,v_2\}$ be a local orthonormal frame in the distribution $\Delta$ with respect to the subriemannian metric $g$ and let $v_0$ be the Reeb field defined by the conditions $\sigma(v_0)=1$ and $d\sigma(v_0,\cdot)=0$. This defines a frame $\{v_0,v_1,v_2\}$ in the tangent bundle $TM$ and we let  $\{\alpha_0=\sigma,\alpha_1,\alpha_2\}$ be the corresponding dual co-frame in the cotangent bundle $T^*M$ (i.e. $\alpha_i(v_j)=\delta_{ij}$).

The frame $\{v_0,v_1,v_2\}$ and the co-frame $\{\alpha_0,\alpha_1,\alpha_2\}$ defined above induce a frame in the tangent bundle $TT^*M$ of the cotangent bundle $T^*M$. Indeed, let $\vec\alpha_i$ be the vector fields on the cotangent bundle $T^*M$ defined by $i_{\vec\alpha_i}\omega=-\alpha_i$. Note that the symbol $\alpha_i$ in the definition of $\vec\alpha_i$ represents the pull back $\pi^*\alpha_i$ of the 1-form $\alpha$ on the manifold $M$ by the projection $\pi:T^*M\to M$. This convention of identifying forms in the manifold $M$ and its pull back on the cotangent bundle $T^*M$ will be used for the rest of this paper without mentioning. Let $h_i:T^*M\to\Real$ be the Hamiltonian lift of the vector fields $v_i$ defined by $h_i(\alpha)=\alpha(v_i)$. Let $\xi_1$ and $\xi_2$ be the 1-forms defined by $\xi_1=h_1\alpha_2-h_2\alpha_1$ and $\xi_2=h_1\alpha_1+h_2\alpha_2$, respectively, and let $\vec\xi_i$ be the vector fields defined by $i_{\vec\xi_i}\omega=-\xi_i$. The vector fields $\vec h_0, \vec h_1,\vec h_2, \vec\sigma, \vec\xi_1, \vec\xi_2$ define a local frame for the tangent bundle $TT^*M$ of the cotangent bundle $T^*M$. In the above notation the subriemannian Hamiltonian is given by $H=\frac{1}{2}((h_1)^2+(h_2)^2)$ and the Hamiltonian vector field is $\vec H=h_1\vec h_1+h_2\vec h_2$.

We also need the bracket relations of the vector fields $v_0,v_1,v_2$. Let $a_{ij}^k$ be the functions on the manifold $M$ defined by

\begin{equation}\label{bracket}
 [v_i,v_j]= a_{ij}^0v_0+a_{ij}^1v_1+a_{ij}^2v_2.
\end{equation}

It is not hard to check that
\begin{equation}\label{coeff}
a_{01}^0=a_{02}^0=0, \quad a_{12}^0=-1, \quad a_{01}^1+a_{02}^2=0.
\end{equation}

Recall that a basis $\{e_1,...,e_n,f_1,...,f_n\}$ in a symplectic vector space with a symplectic form $\omega$ is a Darboux basis if it satisfies $\omega(e_i,e_j)=\omega(f_i,f_j)=0$, and $\omega(f_i,e_j)=\delta_{ij}$. We recall the following theorem from \cite{AgLe2}.

\begin{thm}\label{structural}
For each fixed $\alpha$ in the manifold $T^*M$, there is a moving Darboux frame
\[
e_i(t)=(e^{t\vec H})^*e_i(0),\quad f_i(t)=(e^{t\vec H})^*f_i(0),\quad i=1,2,3
\]
in the symplectic vector space $T_\alpha T^*M$ and functions
\[
R^{11}_t=(e^{t\vec H})^*R^{11}_0,R^{22}_t=(e^{t\vec H})^*R^{22}_0:T^*M\to\Real
\]
depending on time $t$ such that the following structural equations are satisfied
\[
\left\{
  \begin{array}{ll}
    \dot e_1(t)=f_1(t), \\
    \dot e_2(t)=e_1(t), \\
    \dot e_3(t)=f_3(t), \\
    \dot f_1(t)=-R^{11}_te_1(t)-f_2(t),\\
    \dot f_2(t)=-R^{22}_te_2(t),\\
    \dot f_3(t)=0.
  \end{array}
\right.
\]

Moreover,
\[
\left\{
  \begin{array}{ll}
    e_1(0)=\frac{1}{\sqrt{2H}}\vec\xi_1,\\
    e_2(0)=\frac{1}{\sqrt{2H}}\vec\sigma,\\
    e_3(0)=-\frac{1}{\sqrt{2H}}(h_0\vec \alpha_0+h_1\vec \alpha_1+h_2\vec \alpha_2),\\
    f_1(0)=\frac{1}{\sqrt{2H}}[h_1\vec h_2-h_2\vec h_1+\chi_0\vec\alpha_0+(\vec\xi_1h_{12})\vec\xi_1-h_{12}\vec\xi_2],\\
    f_2(0)=\frac{1}{\sqrt{2H}}[2H\vec h_0-h_0\vec H-\chi_1\vec\alpha_0+({\vec\xi_1} a)\vec\xi_1-a\vec\xi_2],\\
    f_3(0)=-\frac{1}{\sqrt{2H}}\vec H,\\
    R^{11}_0=h_0^2+2H\kappa-\frac{3}{2}\vec\xi_1 a,\\
    R^{22}_0=R^{11}_0\vec\xi_1 a-3\vec{H} \vec\xi_1\vec{H} a+3\vec{H}^2\vec\xi_1 a+ \vec\xi_1\vec{H}^2 a.
  \end{array}
\right.
\]
where
\[
\begin{array}{ll}
a=dh_0(\vec H),\\
\chi_0=h_2h_{01}-h_1h_{02}+\vec\xi_1a, \\
\chi_1=h_0a+2{\vec H}\vec\xi_1 a-\vec\xi_1\vec Ha,\\
\kappa=v_1a_{12}^2-v_2a_{12}^1-(a_{12}^1)^2-(a_{12}^2)^2-\frac{1}{2}(a_{01}^2-a_{02}^1),\\
h_{12}(\alpha)=\alpha([v_1,v_2]).
\end{array}
\]
\end{thm}

\smallskip

\section{Connection with the Tanaka-Webster Scalar Curvature}\label{Tanaka}

In this section, we show that $\kappa$ defined in Theorem \ref{structural} coincides with the Tanaka-Webster scalar curvature in CR geometry.

Following \cite{Ta}, we first recall the definition of the Riemannian metric $g^R$ associated to $\alpha_0$. Using the notations in the previous section, the Riemannian metric $g^R$ is defined by the conditions that $\{v_0,v_1,v_2\}$ is orthogonal and the followings hold
\[
g^R(v_0,v_0)=1\quad\text{and} \quad g^R(v_1,v_1)=g^R(v_2,v_2)=\frac{1}{2}.
\]
Let $\mathbf K(v,w)$ be the sectional curvature of the plane spanned by $v$ and $w$ and let $\mathbf{Rc}(v)$ be the Ricci curvature of the vector $v$.

If $\nabla$ denotes the Riemannian connection, then the Tanaka connection $\Tan$ is defined by
\[
\Tan_XY=\nabla_XY + \alpha_0(X)JY -\alpha_0(Y)\nabla_Xv_0+(\nabla_X\alpha_0)(Y)v_0,
\]
where $J$ satisfies $Jv_0=0$, $Jv_1=-v_2$, and $Jv_2=v_1$.

The scalar curvature of the above connection $\Tan$ is called the Tanaka-Webster curvature \cite{Ta,Bl}.

\begin{thm}
The invariant $\kappa$ coincides with the Tanaka-Webster curvature and it satisfies
\[
\kappa=2\mathbf K(v_1,v_2)+\mathbf{Rc}(v_0)+4.
\]
\end{thm}

\begin{proof}
By Koszul's formula, we have the followings
\[
\begin{split}
&\nabla_{v_0}v_0=0, \quad \nabla_{v_1}v_1=\frac{1}{2}a_{01}^1v_0-a_{12}^1v_2,\\
&\nabla_{v_1}v_0=-a_{01}^1v_1-\frac{1}{2}(a_{01}^2+a_{02}^1-2)v_2,\\
&\nabla_{v_1}v_2=\frac{1}{4}(a_{01}^2+a_{02}^1-2)v_0+a_{12}^1v_1,\\
&\nabla_{v_0}v_1=\frac{1}{2}(a_{01}^2-a_{02}^1+2)v_2,\\
&\nabla_{v_2}v_1=\frac{1}{4}(a_{01}^2+a_{02}^1+2)v_0-a_{12}^2v_2,\\
&\nabla_{v_0}v_2=-\frac{1}{2}(a_{01}^2-a_{02}^1+2)v_1\\
&\nabla_{v_2}v_0=-a_{02}^2v_2-\frac{1}{2}(a_{01}^2+a_{02}^1+2)v_1\\
&\nabla_{v_2}v_2=\frac{1}{2}a_{02}^2v_0+a_{12}^2v_1.
\end{split}
\]

If $X$ and $Y$ are contained in the distribution $\Delta$, then
\[
\Tan_XY=\nabla_XY-\alpha_0(\nabla_XY)v_0.
\]
So $\Tan_XY$ is the projection of $\nabla_XY$ onto $\Delta$ in this case. Therefore,
\[
\begin{split}
&\Tan_{v_1}v_1=-a_{12}^1v_2,\quad \Tan_{v_2}v_2=a_{12}^2v_1,\quad \Tan_{v_1}v_2=a_{12}^1v_1,\quad \Tan_{v_2}v_1=-a_{12}^2v_2.
\end{split}
\]

Since $\Tan_Xv_0=-\alpha_0(\nabla_Xv_0)v_0$, we also have $\Tan v_0=0$. It also follows that
\[
\begin{split}
&\Tan_{v_0}v_0=0, \quad \Tan_{v_1}v_1=a_{01}^1v_0-a_{12}^1v_2, \quad \Tan_{v_1}v_0=0.
\end{split}
\]

Finally, we have $\Tan_{v_0}Y=\nabla_{v_0}Y -\alpha_0(\nabla_{v_0}Y)v_0 + JY$. Therefore,
\[
\begin{split}
&\Tan_{v_0}v_1=\frac{1}{2}(a_{01}^2-a_{02}^1)v_2,\quad \Tan_{v_0}v_2=-\frac{1}{2}(a_{01}^2-a_{02}^1)v_1.
\end{split}
\]

We denote the curvature tensor with respect to the connection $\Tan$ by $\Rm$. Since $\Tan v_0=0$, the scalar curvature of $\Rm$ is given by
\[
\begin{split}
2g^R(\Rm(v_2,v_1)v_1,v_2)&=v_1a_{12}^2-v_2a_{12}^1-(a_{12}^1)^2-(a_{12}^2)^2-\frac{1}{2}(a_{01}^2-a_{02}^1)\\
&=\kappa
\end{split}
\]
as claimed.

The assertion $\kappa=2\mathbf K(v_1,v_2)+\mathbf{Rc}(v_0)+4$ follows from a similar calculation. This can also be found in \cite{Bl}.

\end{proof}

\smallskip

\section{Sasakian Space Forms}\label{spaceform}

A three dimensional contact subriemannian manifold is Sasakian if the Reeb field preserves the subriemannian metric. Using the notation of this paper, it is the same as $a=dh_0(\vec H)=0$. A three dimensional Sasakian manifold is a Sasakian space form if the Tanaka-Webster scalar curvature is constant. In this section, we collect various facts about the injectivity domain (see below for the definition) of Sasakian space forms including the recent results in \cite{BoRo}.

Let $(M,\Delta,g)$ be a subriemannian manifold. Let $H$ be the subriemannian Hamiltonian and let $e^{t\vec H}$ be the Hamiltonian flow. Let $\pi:T^*M\to M$ be the projection map and let us fix a point $x$ in the manifold $M$. Let $\Omega_x$ be the set of all covectors $\alpha$ in the cotangent space $T^*_xM$ such that the curve $\gamma:[0,1]\to M$ defined by $\gamma(t)=\pi(e^{t\vec H}(\alpha))$ is a length minimizing geodesic. We call $\Omega=\bigcup_x\Omega_x$ the injectivity domain of the subriemannian manifold. We also let $\Omega^R_x$ be the set of covectors in $\Omega_x$ such that the corresponding curve $\gamma$ has length less than or equal to $R$. A point $\alpha$ in $T^*_xM$ is a cut point if $\gamma(t)=\pi(e^{t\vec H}(\alpha))$ is minimizing geodesic on $[0,1]$ and not minimizing on any larger interval. A point $\alpha$ is a conjugate point if the map $\pi(e^{1\cdot \vec H})$ is singular at $\alpha$.

The Heisenberg group $\mathbb H$ is a well-known example of a Sasakian manifold with vanishing Tanaka-Webster curvature. The manifold in this case is given by $\Real^3$ and the distribution $\Delta$ is the span of two vector fields $\partial_x-\frac{1}{2}y\partial_z$ and $\partial_y+\frac{1}{2}x\partial_z$. These two vector fields also define a subriemannian metric for which they are orthonormal. In this case all cut points are conjugate points and $\Omega^R$ is given by
\begin{equation}\label{Hcut}
\Omega^R_0=\{\alpha|\sqrt{2H(\alpha)}\leq R,-2\pi\leq h_0(\alpha)\leq 2\pi\}.
\end{equation}

Recall that $SU(2)$, the special unitary group, consists of $2\times2$ matrices with complex coefficients and determinant 1. The Lie algebra $su(2)$ consists of skew Hermitian matrices with trace zero. The left invariant vector fields of the following two elements in $su(2)$
\[
u_1=\left(
\begin{array}{cc}
0 & 1/2\\
-1/2 & 0
\end{array}\right),\quad u_2=\left(
\begin{array}{cc}
0 & i/2\\
i/2 & 0
\end{array}\right)
\]
span the standard distribution $\Delta$ on $SU(2)$. Let $g^c$ be the subriemannian metric for which $g^c(cu_1,cu_2)=1$. The Reeb field in this case is $c^2u_0$, where
\[
u_0=\left(
\begin{array}{cc}
i/2 & 0\\
0 & -i/2
\end{array}\right).
\]
A computation shows that the Tanaka-Webster curvature is given by $c^2$. It follows from the result in \cite{BoRo} that all cut points are conjugate points in this case and $\Omega^R$ is given by
\begin{equation}\label{SU2cut}
\Omega^R_{id}=\{\alpha|\sqrt{h_0(\alpha)^2+2c^2H(\alpha)}\leq 2\pi\}.
\end{equation}

\begin{rem}
The notations in here and that of \cite{BoRo} are slightly different. It was shown in \cite[Theorem 12]{BoRo} that all cut points are conjugate points in the case $SU(2)$ with $c=1$. Moreover, the path $t\mapsto\pi(e^{t\vec H}(\alpha))$ hits the first conjugate point at the time $\frac{2\pi}{\sqrt{1+h_0(\alpha)^2}}$, where $H(\alpha)=\frac{1}{2}$. This is equivalent to (\ref{SU2cut}). The rest of the cases with $c\neq 1$ follow from scaling.
\end{rem}

The special linear group $SL(2)$ is the set of all $2\times 2$ matrices with real coefficients and determinant 1. The Lie algebra $sl(2)$ is the set of all $2\times 2$ real matrices with trace zero. The left invariant vector fields of the following two elements in $sl(2)$
\[
u_1=\left(
\begin{array}{cc}
1/2 & 0\\
0 & -1/2
\end{array}\right),\quad u_2=\left(
\begin{array}{cc}
0 & 1/2\\
1/2 & 0
\end{array}\right)
\]
span the standard distribution $\Delta$ on $SL(2)$. Let $g^c$ be the subriemannian metric for which $g^c(cu_1,cu_2)=1$. The Reeb field in this case is $c^2u_0$, where
\[
u_0=\left(
\begin{array}{cc}
0 & -1/2\\
1/2 & 0
\end{array}\right).
\]
The Tanaka-Webster curvature is given by $-c^2$. The structure of the set of cut points in this case is much more complicated. However, the result in \cite{BoRo} and a computation shows the following.

\begin{thm}\label{SL2cut}
Assume that a cut point $\alpha$ in the cotangent bundle of $SL(2)$ with subriemannian metric $g^c$ is contained in $\Omega^R_{id}$, where $R=\frac{2\sqrt 2\,\pi}{c}$. Then it is a conjugate point. Moreover, it satisfies
\[
\sqrt{|h_0(\alpha)^2-2c^2H(\alpha)|}=2\pi
\]
\end{thm}

\begin{proof}
Let $\tau=h_0^2-2Hc^2$. From the proof of \cite[Section 5.3, Theorem 15]{BoRo}, $\alpha$ in the cotangent space $T^*_xSL(2)$ at a point $x$ is a cut point but not a conjugate point only if it satisfies

\begin{equation}\label{tan1}
\frac{\tan(h_0(\alpha)/2)}{h_0(\alpha)}=\frac{\tanh(\sqrt{-\tau(\alpha)}/2)}{\sqrt{-\tau(\alpha)}}
\end{equation}
for $\tau(\alpha)<0$,
\begin{equation}\label{tan2}
\frac{\tan(h_0(\alpha)/2)}{h_0(\alpha)}=\frac{\tan(\sqrt{\tau(\alpha)}/2)}{\sqrt{\tau(\alpha)}}
\end{equation}
for $\tau(\alpha)>0$, or
\begin{equation}\label{tan3}
\frac{\tan(h_0(\alpha)/2)}{h_0(\alpha)}=\frac{1}{2}
\end{equation}
for $\tau(\alpha)=0$.

Let $r_1,r_2,r_3$ be the infimum of $2H(\alpha)c^2$ where $\alpha$ runs over positive solutions of (\ref{tan1}), (\ref{tan2}),(\ref{tan3}), respectively. The goal is to find the minimum of $\{r_1,r_2,r_3\}$.

Let $f(x)=\frac{\tan(\sqrt x)}{\sqrt x}$ and $g(x)=\frac{\tanh(\sqrt x)}{\sqrt x}$. Let $F_1$ be a branch of inverses of $f\Big|_{[0,\infty)}$ and let $G$ be the inverse of $g\Big|_{[0,\infty)}$. Finding $r_1$ is the same as minimizing $4(F_1+G)$.

A computation shows that that the derivatives of $F_1$ and $G$ satisfy
\begin{equation}\label{F1G}
F_1'(x)=\frac{1+x^2F_1(x)-x}{2F_1(x)}, \quad G'(x)=\frac{1-x^2G(x)-x}{2G(x)}.
\end{equation}
Since $F_1$ and $G$ are nonnegative, $F_1'+G'=0$ implies that $x=1$. It follows that $r_1\geq r_3$.

Let $F_2$ be another branch of inverses of $f\Big|_{[0,\infty)}$ for which $F_2>F_1$. We can assume that $F_1$ is the smallest branch and $F_2$ is the second smallest branch. In this case, finding $r_2$ is the same as minimizing $4(F_2-F_1)$. It follows from (\ref{F1G}) that $F_2'(x)-F_1'(x)=0$ implies $x=1$. Therefore, there are two possibilities. Either the minimum of $F_2-F_1$ occurs at $x=1$ which implies that $r_2\geq r_3$ or $4(F_2-F_1)$ goes to the infimum as $x\to\infty$ which implies that $r_2=4\left(\left(\frac{3\pi}{2}\right)^2-\left(\frac{\pi}{2}\right)^2\right)=8\pi^2<r_3$.

The last assertion follows from \cite{BoRo}.
\end{proof}

\smallskip

\section{Volume of Subriemannian Balls}\label{volestimate}

In this section, we give an estimate on the volume of subriemannian balls of Sasakian manifolds assuming the Tanaka-Webster curvature is bounded below. More precisely, let us fix a point $x$ in the manifold and let $v_1,v_2$ be an orthonormal basis of the subriemannian metric around $x$. Let $v_0$ be the Reeb field and $\alpha_0,\alpha_1,\alpha_2$ be the dual coframe of the frame $v_0,v_1,v_2$ (i.e. $\alpha_i(v_j)=\delta_{ij}$). We use this coframe to introduce coordinates on the cotangent space $T_x^*M$ and let $\mathfrak m$ be the corresponding volume form. We will also denote the corresponding Lebesgue measure by the same symbol. Let $(r,\theta,h)$ be the cylindrical coordinates on $T^*_xM$ corresponding to the above coordinate system (i.e. $h=h_0(\alpha)$, $r^2=2H(\alpha)$, and $\tan(\theta)=\frac{h_2(\alpha)}{h_1(\alpha)}$). Recall that $\Omega_R$ denotes the set of all covectors $\alpha$ such that $\sqrt{2H(\alpha)}\leq R$ and the curve $t\mapsto \pi(e^{t\vec H}(\alpha)), 0\leq t\leq 1$ is length minimizing. We use the coordinate system introduced above on $T^*_xM$ to identify the set $\Omega_R$ with a subset in $\Real^n$. Finally, let $\eta$ be the volume form defined by the condition $\eta(v_0,v_1,v_2)=1$. We denote the measure induced by $\eta$ using the same symbol.

\begin{thm}\label{volgrowth}
Assume that there exists a constant $k_1$ (resp. $k_2$) such that the Tanaka-Webster curvature $\kappa$ of a three dimensional Sasakian manifold satisfies $\kappa\geq k_1$ (resp. $\leq k_2$) on the ball $B(x,R)$ of radius $R$ centered at the point $x$. Then
\[
\eta(B(x,R))\leq \int_{\Omega_R}b^{k_1}d\mathfrak m \quad \left(\text{resp.} \geq \int_{\Omega_R}b^{k_2}d\mathfrak m\right ),
\]
where $b_k:T^*_xM\to\Real$ is defined via the above mentioned cylindrical coordinates by
\[
b^k=
\begin{cases}
\frac{r^2(2-2\cos(\tau)-\tau\sin(\tau))}{\sigma^2} & \text{if}\quad  \sigma>0,\\
\frac{r^2(2-2\cosh(\tau)+\tau\sinh(\tau)))}{\sigma^2} & \text{if}\quad \sigma <0,\\
\frac{r^2}{12} & \text{if} \quad \sigma=0,
\end{cases}
\]
$\sigma=h^2+r^2k$, and $\tau=\sqrt{|\sigma|}$.
\end{thm}

As a corollary, we have a formula for the volume of subriemannian balls on Sasakian space forms. Remark that explicit formula for the set $\Omega_R$ in various examples are present in Section \ref{spaceform} (see also \cite{BoRo} for more details).

\begin{cor}\label{constantvol}
Assume that the three dimensional subriemannian manifold is a Sasakian space form with Tanaka-Webster curvature $k$. Then
\[
\eta(B(x,R))= \int_{\Omega_R}b^{k}d\mathfrak m.
\]
\end{cor}

\begin{proof}[Proof of Theorem \ref{volgrowth}]
Recall that $\eta$ is the measure on $M$ defined by $\eta(v_0,v_1,v_2)=1$. Let $\psi_t:T^*_xM\to M$ be the map $\psi_t(\alpha)=\pi(e^{t\cdot \vec H}(\alpha))$ and let $\rho_t:T^*_xM\to\Real$ be the function defined by
\begin{equation}\label{rho}
\psi_t^*\eta=\rho_t\mathfrak m.
\end{equation}

Let us fix a covector $\alpha$ in $T^*_xM$. Let $e_1(t),e_2(t),e_3(t),f_1(t),f_2(t),f_3(t)$ be a canonical Darboux frame at $\alpha$ defined by Theorem \ref{structural}. Let $a_{ij}(t)$ and $b_{ij}(t)$ be defined by
\begin{equation}\label{ei}
e_i(0)=\sum_{j=1}^3(a_{ij}(t)e_j(t)+b_{ij}(t)f_j(t)).
\end{equation}
Finally let $A_t$ and $B_t$ be the matrices with $(i,j)$-th entry equal to $a_{ij}(t)$ and $b_{ij}(t)$, respectively.

By definition of $\mathfrak m$, we have $\mathfrak m(e_1(0),e_2(0),e_3(0))=\frac{1}{\sqrt{2H(\alpha)}}$. It also follows from Theorem \ref{structural} that $\eta(f_1(0),f_2(0),f_3(0))=\sqrt{2H(\alpha)}$. Therefore, (\ref{rho}) implies that
\begin{equation}\label{measure}
\rho_t=2H(\alpha)\det B_t.
\end{equation}

It follows that from (\ref{measure}) that
\begin{equation}\label{balleqn}
\eta(B(x,r))=\int_{\psi_1(\Omega_r)}d\eta=\int_{\Omega_r}|\rho_1|d\mathfrak m=\int_{\Omega_r}2H(\alpha)|\det B_1|d\mathfrak m.
\end{equation}

Let $E_t=(e_1(t),e_2(t),e_3(t))^T$ and let $F_t=(f_1(t),f_2(t),f_3(t))^T$. Here the superscript $T$ denote matrix transpose. By the definition of the matrices $A_t$ and $B_t$, we have
\[
E_0=A_tE_t+B_tF_t.
\]

If we differentiate the above equation with respect to time $t$, we have
\[
\begin{split}
0&=\dot A_tE_t+A_t\dot E_t+\dot B_tF_t+B_t\dot F_t\\
&=\dot A_tE_t+A_tC_1E_t+A_tC_2F_t+\dot B_tF_t+B_t(-R_tE_t-C_1^TF_t),
\end{split}
\]

Since the manifold is Sasakian, $a=dh_0(\vec H)=0$. Therefore, by Theorem \ref{structural}, $E_t$ and $F_t$ satisfy the following equatoins
\[
\dot E_t=C_1E_t+C_2F_t,\quad F_t=-R_tE_t-C_1^TF_t,
\]
where
\[
C_1=\left(\begin{array}{ccc}
0 & 0 & 0\\
1 & 0 & 0\\
0 & 0 & 0
\end{array}\right), C_2=\left(\begin{array}{ccc}
1 & 0 & 0\\
0 & 0 & 0\\
0 & 0 & 1
\end{array}\right),
\]
\[
R_t=\left(\begin{array}{ccc}
R^{11}_t & 0 & 0\\
0 & 0 & 0\\
0 & 0 & 0
\end{array}\right)=\left(\begin{array}{ccc}
h_0^2(\alpha)+2\kappa_tH(\alpha) & 0 & 0\\
0 & 0 & 0\\
0 & 0 & 0
\end{array}\right),
\]
and $\kappa_t=\kappa(\psi_t(\alpha))$.

It follows that the matrices $A_t$ and $B_t$ satisfy the following equations
\begin{equation}\label{AB}
\dot A_t+A_tC_1-B_tR_t=0,\quad \dot B_t+A_tC_2-B_tC_1^T=0
\end{equation}
with initial conditions $B_0=0$ and $A_0=Id$.

If we set $S_t=B_t^{-1}A_t$ and $U_t=S_t^{-1}=A_t^{-1}B_t$, then they satisfy the following Riccati equations.
\[
\dot S_t-S_tC_2S_t+C_1^TS_t+SC_1-R_t=0
\]
and
\[
\dot U_t+U_tR_tU_t-C_1U_t-U_tC_1^T+C_2=0
\]
with initial condition $U_0=0$.

Let us fix a constant $k$ and consider the following Riccati equation with constant coefficients
\begin{equation}\label{new}
\dot U^k_t+U^k_tR^kU^k_t-C_1U_t^k-U_t^kC_1^T+C_2=0
\end{equation}
and initial condition $U^k_0=0$, where $R^k=\left(
\begin{array}{ccc}
h_0^2(\alpha)+2kH(\alpha) & 0 & 0 \\
0 & 0 & 0 \\
0 & 0 & 0	
\end{array}
\right)$.

The solution of (\ref{new}) can be found by the method in \cite{Le}. If $h_0^2+2H(\alpha)k>0$, then
\[
U^k_t=\left(
\begin{array}{ccc}
\frac{-t\sin(\tau_t)}{\tau_t\cos(\tau_t)} & \frac{t^2(\cos(\tau_t)-1)}{\tau_t^2\cos(\tau_t)} &  0\\
\frac{t^2(\cos(\tau_t)-1)}{\tau_t^2\cos(\tau_t)} & \frac{t^3(\tau_t\cos(\tau_t)-\sin(\tau_t))}{\tau_t^3\cos(\tau_t)} & 0\\
0 & 0 & -t
\end{array}\right).
\]

If $h_0^2+2H(\alpha)k<0$, then
\[
U^k_t=\left(
\begin{array}{ccc}
\frac{-t\sinh(\tau_t)}{\tau_t\cosh(\tau_t)} & \frac{t^2(1-\cosh(\tau_t))}{\tau_t^2\cosh(\tau_t)} &  0\\
\frac{t^2(1-\cosh(\tau_t))}{\tau_t^2\cosh(\tau_t)} & \frac{t^3(\sinh(\tau_t)-\tau_t\cosh(\tau_t))}{\tau_t^3\cosh(\tau_t)} & 0\\
0 & 0 & -t
\end{array}\right).
\]

If $h_0^2+2H(\alpha)k=0$, then
\[
U^k_t=\left(
\begin{array}{ccc}
-t & -\frac{t^2}{2} &  0\\
-\frac{t^2}{2} & -\frac{t^3}{3} & 0\\
0 & 0 & -t
\end{array}\right),
\]
where $\tau_t=t\sqrt{|h_0^2+2H(\alpha)k|}$.

If we call the inverse $S^k_t=(U^k_t)^{-1}$, then
\[
S^k_t=\left(
\begin{array}{ccc}
\frac{\tau_t(\tau_t\cos(\tau_t)-\sin(\tau_t))}{t(2-2\cos(\tau_t)-\tau_t\sin(\tau_t))} & \frac{\tau_t^2(1-\cos(\tau_t))}{t^2(2-2\cos(\tau_t)-\tau_t\sin(\tau_t))} & 0 \\
\frac{\tau_t^2(1-\cos(\tau_t))}{t^2(2-2\cos(\tau_t)-\tau_t\sin(\tau_t))} & \frac{-\tau_t^3\sin(\tau_t)}{t^3(2-2\cos(\tau_t)-\tau_t\sin(\tau_t))} & 0 \\
0 & 0 & -\frac{1}{t}
\end{array}\right),
\]
if $h_0^2+2H(\alpha)k>0$.

\[
S^k_t=\left(
\begin{array}{ccc}
\frac{\tau_t(\sinh(\tau_t)-\tau_t\cosh(\tau_t))}{t(2-2\cosh(\tau_t)+\tau_t\sinh(\tau_t))} & \frac{\tau_t^2(\cosh(\tau_t)-1)}{t^2(2-2\cosh(\tau_t)+\tau_t\sinh(\tau_t))} & 0 \\
\frac{\tau_t^2(\cosh(\tau_t)-1)}{t^2(2-2\cosh(\tau_t)+\tau_t\sinh(\tau_t))} & \frac{-\tau_t^3\sinh(\tau_t)}{t^3(2-2\cosh(\tau_t)+\tau_t\sinh(\tau_t))} & 0 \\
0 & 0 & -\frac{1}{t}
\end{array}\right),
\]
if $h_0^2+2H(\alpha)k<0$.

\[
S^k_t=\left(
\begin{array}{ccc}
-\frac{4}{t} & \frac{6}{t^2} & 0 \\
\frac{6}{t^2} & -\frac{12}{t^3} & 0 \\
0 & 0 & -\frac{1}{t}
\end{array}\right),
\]
if $h_0^2+2H(\alpha)k=0$.

By \cite{Ro}, if $k_1\leq\kappa\leq k_2$, then
\begin{equation}\label{Ucompare}
U^{k_2}_t\leq U_t\leq U^{k_1}_t\leq 0.
\end{equation}
Therefore, $S^{k_1}_t\leq S_t\leq S^{k_2}_t$.

On the other hand, by (\ref{AB}) and the definition of $S_t$, we have
\[
\dot B_t+B_t(S_tC_2-C_1^T)=0.
\]

It follows that $\frac{d}{dt}\det B_t=\text{tr}(C_1^T-S_tC_2)\det B_t=-\text{tr}(S_tC_2)\det B_t$. If we replace the matrix $R_t$ in (\ref{AB}) by $R^k$ and denote the solution by $A^k_t$ and $B^k_t$, then we have
\[
\frac{\frac{d}{dt}\det B_t}{\det B_t}=-\text{tr}(S_tC_2)\geq -\text{tr}(S_t^{k_2}C_2)=\frac{\frac{d}{dt}\det B_t^{k_2}}{\det B_t^{k_2}}.
\]

It follows that $\frac{\det B_t}{\det B_t^{k_2}}$ is nondecreasing. By definition of $U_t$ and (\ref{Ucompare}), we also have
\[
\lim_{t\to 0}\frac{\det B_t}{\det B_t^{k_2}}=\lim_{t\to 0}\frac{\det A_t\det U_t}{\det A_t^{k_2}\det U_t^{k_2}}\geq \lim_{t\to 0}\frac{\det A_t}{\det A_t^{k_2}}=1
\]

Therefore, it follows that
\[
\frac{\det B_t}{\det B_t^{k_2}}\geq \lim_{t\to 0}\frac{\det B_t}{\det B_t^{k_2}}\geq 1.
\]

Similarly, we also have
\[
\frac{\det B_t}{\det B_t^{k_1}}\leq  1.
\]

A calculation gives
\begin{equation}\label{btk1}
|\det B_t^k|:=b_t^k=\frac{t(2-2\cos(\tau_t)-\tau_t\sin(\tau_t))}{\tau_1^4}
\end{equation}
if $h_0^2+2H(\alpha)k>0$,
\begin{equation}\label{btk2}
b_t^k=\frac{t(2-2\cosh(\tau_t)+\tau_t\sinh(\tau_t))}{\tau_1^4}
\end{equation}
if $h_0^2+2H(\alpha)k<0$, and
\begin{equation}\label{btk3}
b_t^k=\frac{t^5}{12}
\end{equation}
if $h_0^2+2H(\alpha)k=0$.

It follows that $b_t^{k_2}\leq |\det B_t|\leq b_t^{k_1}$. Therefore, we have the following as claimed
\[
\int_{\Omega_R}r^2b_1^{k_2}d\mathfrak m\leq\eta(B(x,r))\leq \int_{\Omega_R}r^2b_1^{k_1}d\mathfrak m.
\]
\end{proof}

\smallskip

\section{Subriemannian Bishop Theorem}\label{subBishop}

In this section, we prove a subriemannian analog of Bishop theorem for three dimensional Sasakian manifolds. Recall that $\eta$ is the volume form defined by the condition $\eta(v_0,v_1,v_2)$. We denote the measure induced by $\eta$ using the same symbol and let $\eta^k$ be the corresponding measure in a Sasakian space form of curvature $k$. Let $B^k(R)$ be a subriemannian ball of radius $R$ in one of the Sasakian space forms $SU(2)$, $\mathbb H$, or $SL(2)$ of curvature $k$ (see Section \ref{spaceform} for a discussion of these space forms).

\begin{thm}(Subriemannian Bishop Theorem)\label{Bishop}
Assume that the Tanaka-Webster scalar curvature $\kappa$ of a three dimensional Sasakian manifold satisfies $\kappa\geq k$ on the ball $B(x,R)$ for some constant $k$. If $k\geq 0$, then
\begin{equation}\label{compareball}
\eta(B(x,R))\leq\eta^k(B^k(R))
\end{equation}
and equality holds only if $\kappa=k$ on $B(x,R)$. The same conclusion holds for $k<0$ provided that $R\leq \frac{2\sqrt 2\,\pi}{c}$.
\end{thm}

\begin{proof}
Let us start with the proof of (\ref{compareball}). By (\ref{Hcut}), (\ref{SU2cut}), Theorem \ref{SL2cut}, and Corollary \ref{constantvol}, it is enough to show that $\Omega_R$ is contained in the set
\[
\{\alpha\in T^*_xM|\sqrt{|h_0(\alpha)^2+kH(\alpha)|}\leq 2\pi\}
\]

Suppose that there is a covector $\alpha$ in $\Omega_R$ such that
\[
\tau_1:=\sqrt{|h_0(\alpha)^2+kH(\alpha)|}> 2\pi.
\]
Using the notation in the proof of Theorem \ref{volgrowth}, we let
\[
E_0=A_tE_t+B_tF_t,
\]
where $E_t=(e_0(t),e_1(t),e_2(t))^T$ and $F_t=(f_0(t),f_1(t),f_2(t))^T$ are canonical Darboux frame at $\alpha$.

By the proof of Theorem \ref{volgrowth}, we have that $|\det B_t|\leq b_t^k$, where $b_t^k$ is defined in (\ref{btk1}), (\ref{btk2}), and (\ref{btk3}). Since $\tau_1>2\pi$, it follows that $\det B_t=0$ for some $t<1$. Therefore, $t\alpha$ is a conjugate point contradicting the fact that $\alpha$ is contained in $\Omega_R$.

Next, suppose equality holds in (\ref{compareball}) and $\kappa> k$ on an open set $\mathcal O$ contained in the ball $B(x,R)$. For each point $y$ in $\mathcal O$, let $\gamma(t)=\pi(e^{t\vec H}(\alpha))$ be a minimizing geodesic connecting $x$ and $y$. It follows that  $R_t>R_t^k$ for all $t$ close enough to $1$.  By the result in \cite{Ro} and a similar argument as in Theorem \ref{volgrowth}, we have $|\det B_1|<|\det B_1^{k}|$. It follows from (\ref{balleqn}) that $\eta(B(x,R))<\eta^k(B^k(R))$ which is a contradiction.
\end{proof}

\smallskip

\section{Subriemannian Hessian and Laplacian}\label{HessLaplace}

In this section, we introduce subriemannian versions of Hessian and Laplacian. For the computation, we will also give an expression for it in the canonical Darboux frame.

Assume that a functon $f:M\to\Real$ is twice differentiable at a point $x$ in the manifold $M$. The canonical Darboux frame
\[
\{e_1(t),e_2(t),e_3(t),f_1(t),f_2(t),f_3(t)\}
\]
at $df_x$ gives a splitting of the tangent space $T_{df_x}T^*M=\mathcal H\oplus\mathcal V$ defined by
\[
\mathcal H=\textbf{span}\{f_1(0),f_2(0),f_3(0)\}, \quad \mathcal V=\textbf{span}\{e_1(0),e_2(0),e_3(0)\}.
\]

The differential $d\pi$ of the projection map $\pi:T^*M\to M$ defines an identification between $\mathcal H$ and $T_xM$. On the other hand, the map $\iota:T_x^*M\to\mathcal V$ defined by $\iota(\alpha)=-\vec\alpha$ also gives an identification between $T^*_xM$ and $\mathcal V$.

Let us consider the differential $d(df_x):T_xM\to T_{df_x}T^*M$ at $x$ of the map $x\mapsto df_x$. It defines a three dimensional subspace $\Lambda:=d(df_x)(T_xM)$ of the tangent space $T_{df_x}T^*M$. Since $\Lambda$ is transvesal to the space $\mathcal V$, it defines a linear map $S$ from $\mathcal H$ to $\mathcal V$ for which the graph is given by $\Lambda$. More precisely, if $w=w_h+w_v$ is a vector in the space $\Lambda$, where $w_h$ and $w_v$ are in $\mathcal H$ and $\mathcal V$, respectivly. Then $S(w_h)=w_v$. Under the identifications of the tangent space $T_xM$ with $\mathcal H$ and the cotangent space $T^*_xM$ with $\mathcal V$, we obtain a linear map $H^{SR}f(x):T_xM\to T^*_xM$ called subriemannian Hessian. More precisely,
\[
H^{SR}f(x)(d\pi(w))=\iota^{-1}Sw.
\]

\begin{prop}\label{HessianSym}
The subriemannian Hessian $H^{SR}f$ is symmetric i.e. $\left<H^{SR}f(x)v,w\right>=\left<H^{SR}f(x)w,v\right>$.
\end{prop}

\begin{proof}
Let $w_1$ and $w_2$ be two vectors in the subspace $\mathcal H$. Since the subspace $\Lambda$ is a Lagrangian subspace, we have
\[
\omega(w_1+S(w_1),w_2+S(w_2))=0.
\]

Since both $\mathcal H$ and $\mathcal V$ are Lagrangian subspaces, we also have
\[
\omega(w_1,S(w_2))+\omega(S(w_1),w_2)=0.
\]

It follows from skew symmetry of $\omega$ and the definition of subriemannian Hessian that
\[
\begin{split}
&\left<H^{SR}f(x)d\pi(w_2),d\pi(w_1)\right>\\
&=\omega(w_1,S(w_2))\\
&=\omega(w_2,S(w_1))\\
&=\left<H^{SR}f(x)d\pi(w_1),d\pi(w_2)\right>.
\end{split}
\]
\end{proof}

Let $\iota_i=d\pi(f_i(0))$ and let $\mathfrak Hf(x)$ be the subriemannian Hessian matrix with $ij$-th entry $\mathfrak H_{ij}f(x)$ defined by
\[
\mathfrak H_{ij}f(x)=\left<H^{SR}f(x)\iota_i,\iota_j\right>=\omega(Sf_i(0),f_j(0)).
\]

Recall that $v_0$ denotes the Reeb field and $v_1,v_2$ be an orthonormal basis with respect to the subriemannian metric $g$.

\begin{prop}\label{Hessian}
The subriemannian Hessian matrix $\mathfrak Hf$ satisfies the following
\[
\begin{split}
\mathfrak H_{11}f&=\frac{(v_1f)^2v_2^2f+(v_2f)^2v_1^2f-(v_1f)(v_2f)(v_1v_2f+v_2v_1f)}{(v_1f)^2+(v_2f)^2}\\
&+a_{12}^1v_2f-a_{12}^2v_1f,\\
\mathfrak H_{12}f&=(v_1f)v_2v_0f-(v_2f)v_1v_0f+(v_0f)\mathfrak H_{13}f+\vec\xi_1a(df), \\
\mathfrak H_{13}f&=\frac{(v_1f)(v_2f)(v_1^2f-v_2^2f)-(v_1f)^2(v_2v_1f)+(v_2f)^2(v_1v_2f)}{(v_1f)^2+(v_2f)^2},\\
\mathfrak H_{22}f&=((v_1f)^2+(v_2f)^2)v_0^2f-(v_0f)(v_1f)v_1v_0f-(v_0f)(v_2f)v_2v_0f\\
&+v_0f(\mathfrak H_{23}f+a(df))-\chi_1(df)\\
\mathfrak H_{23}f&=(v_0f)\mathfrak H_{33}f-(v_1f)v_0v_1f-(v_2f)v_0v_2f,\\
\mathfrak H_{33}f&=\frac{(v_1f)^2v_1^2f+(v_1f)(v_2f)(v_2v_1f+v_1v_2f)+(v_2f)^2v_2^2f}{(v_1f)^2+(v_2f)^2}.
\end{split}
\]
\end{prop}

\begin{proof}
Let $A$ be the matrix with the $ij$-th entry $a_{ij}$ defined by
\[
d(df_x)(\iota_i)=f_i(0)+\sum_{k=1}^3a_{ik}e_k(0),
\]

By the definition of the linear map $S$, we have
\[
S(f_i(0))=\sum_{j=1}^3a_{ij}e_j(0).
\]

It follows that
\[
\mathfrak H_{ij}f(x)=\omega(S(f_i(0)),f_j(0))=-a_{ij}.
\]

Let us look at the case $a_{11}$. We have $2H(df)=(v_1f)^2+(v_2f)^2$. Since $\pi(df_x)=x$, we also have
\[
\pi^*\alpha_i(d(df_x))(\iota_1)=\alpha_i(\iota_1).
\]
Therefore, by Theorem \ref{structural}, we have
\[
\begin{split}
\mathfrak H_{11}f(x)&=-a_{11}\\
&=-\omega(f_1(0),d(df_x)(\iota_1))\\
&=\frac{(v_1f)^2v_2^2f+(v_2f)^2v_1^2f-(v_1f)(v_2f)(v_1v_2f+v_2v_1f)}{(v_1f)^2+(v_2f)^2}\\
&+a_{12}^1v_2-a_{12}^2v_1.
\end{split}
\]

Similar calculations give the rest of the entries of $A$.
\end{proof}

We define the horizontal gradient $\nabla_H$ by $g(\nabla_Hf,v)=df(v)$ for all vectors $v$ in the distribution $\Delta$. Recall that $\eta$ is the volume form defined by $\eta(v_0,v_1,v_2)=1$. The sub-Laplacian $\Delta_H$ is defined by $\Delta_Hf=\textbf{div}_\eta\nabla_H f$. Here $\textbf{div}_\eta$ denotes the divergence with respect to the volume form $\eta$. Let $C_2$ be the matrix defined by
\[
C_2=\left(\begin{array}{ccc}
1 & 0 & 0\\
0 & 0 & 0\\
0 & 0 & 1
\end{array}\right).
\]

\begin{cor}\label{HessLaplaceRelate}
The subriemannian Hessian matrix $\mathfrak H$ and the sub-Laplacian satisfies
\[
\textbf{tr}(C_2\mathfrak Hf)=\Delta_Hf,
\]
where $\textbf{tr}$ denote the trace of the matrix.
\end{cor}

\begin{proof}
A simple calculation using Proposition \ref{Hessian} shows that
\[
\textbf{tr}(C_2\mathfrak Hf)=(v_1^2+v_2^2+a_{12}^1v_2-a_{12}^2v_1)f=\Delta_Hf.
\]
\end{proof}

Let $d$ be the subriemannian distance function of a subriemannian manifold  $(M,\Delta,g)$. Let us fix a point $x_0$ in the manifold $M$. Let $r:M\to\Real$ be the function $r(x)=d(x,x_0)$ and let $\ds (x)=-\frac{1}{2}r^2(x)$. Finally, we show that the subriemannian Hessian of the function $\ds$ takes a very simple form.

\begin{prop}\label{Hessiands}
The subriemannian Hessian matrix $\mathfrak H\ds$ satisfies the following wherever $\ds$ is twice differentiable
\[
\begin{split}
\mathfrak H_{11}\ds&=\Delta_H\ds+1,\\
\mathfrak H_{12}\ds&=(v_1\ds)v_2v_0\ds-(v_2\ds)v_1v_0\ds+\vec\xi_1a(d\ds), \\
\mathfrak H_{13}\ds&=0,\\
\mathfrak H_{22}\ds&=-2\ds\, v_0^2\ds+(v_0\ds)^2-\chi_1(d\ds),\\
\mathfrak H_{23}\ds&=0,\\
\mathfrak H_{33}\ds&=-1.
\end{split}
\]

If we assume that the subriemannian manifold is Sasakian, then the above simplifies to
\[
\begin{split}
\mathfrak H_{11}\ds&=\Delta_H\ds+1,\\
\mathfrak H_{12}\ds&=(v_1\ds)v_2v_0\ds-(v_2\ds)v_1v_0\ds, \\
\mathfrak H_{13}\ds&=0,\\
\mathfrak H_{22}\ds&=-2\ds\, v_0^2\ds+(v_0\ds)^2,\\
\mathfrak H_{23}\ds&=0,\\
\mathfrak H_{33}\ds&=-1.
\end{split}
\]
\end{prop}

\begin{proof}
The first formula follows from differentiating the following equation by $v_0$, $v_1$, and $v_2$
\[
(v_1\ds)^2+(v_2\ds)^2=-2\ds
\]
and combining them with Proposition \ref{Hessian}.

The second follows from $a=0$.
\end{proof}

\smallskip

\section{Sub-Laplacian of Distance Functions in Sasakian Space Forms}\label{LaplaceSpaceForm}

In this section, we give a formula for the sub-Laplacian of the subriemannian distance function of a Sasakian space form. Let $d$ be the subriemannian distance function of a subriemannian manifold  $(M,\Delta,g)$. Let us fix a point $x_0$ in the manifold $M$. Let $r:M\to\Real$ be the function $r(x)=d(x,x_0)$ and let $\ds (x)=-\frac{1}{2}r^2(x)$.

\begin{thm}\label{HessExplicit}
Assume that the subriemannian manifold $(M,\Delta,g)$ is a three dimensional Sasakian space form of Tanaka-Webster curvature $k$. Then the subriemannian Hessian matrix $\mathfrak H\ds$ satisfies the following wherever $\ds$ is twice differentiable.
\[
\mathfrak H\ds(z) =
\begin{cases}
-\left(
\begin{array}{ccc}
\frac{\tau(\sin\tau-\tau\cos\tau)}{2-2\cos\tau-\tau\sin\tau} & \frac{\tau^2(1-\cos\tau)}{2-2\cos\tau-\tau\sin\tau} & 0 \\
\frac{\tau^2(1-\cos\tau)}{2-2\cos\tau-\tau\sin\tau} & \frac{\tau^3\sin\tau}{2-2\cos\tau-\tau\sin\tau} & 0 \\
0 & 0 & 1
\end{array}\right) & \text{if } \sigma(z)>0,\\
-\left(
\begin{array}{ccc}
\frac{\tau(\tau\cosh\tau-\sinh\tau)}{2-2\cosh\tau+\tau\sinh\tau} & \frac{\tau^2(\cosh\tau-1)}{2-2\cosh\tau+\tau_0\sinh\tau} & 0 \\
\frac{\tau^2(\cosh\tau-1)}{2-2\cosh\tau+\tau\sinh\tau} & \frac{\tau^3\sinh\tau}{2-2\cosh\tau+\tau\sinh\tau} & 0 \\
0 & 0 & 1
\end{array}\right) & \text{if } \sigma(z)<0,\\
-\left(
\begin{array}{ccc}
4 & 6 & 0 \\
6 & 12 & 0 \\
0 & 0 & 1
\end{array}\right) & \text{if } \sigma(z)=0,
\end{cases}
\]
where $\sigma(z)=(v_0\ds(z))^2-2\ds(z)k$ and $\tau(z)=\sqrt{|\sigma(z)|}$.
\end{thm}

By combining Theorem \ref{HessExplicit} and Proposition \ref{HessLaplaceRelate}, we obtain the following.

\begin{cor}\label{LaplaceExplicit}
Let $d$ be the subriemannian distance function of a three dimensional Sasakian space form of Tanaka-Webster curvature $k$. Then the sub-Laplacian $\Delta_Hr$ satisfies the following wherever $r$ is twice differentiable.
\[
\Delta_H r(z)=
\begin{cases}
\frac{\tau(\sin\tau-\tau\cos\tau)}{r(2-2\cos\tau-\tau\sin\tau)} & \text{if } \sigma(z)>0,\\
\frac{\tau(\tau\cosh\tau-\sinh\tau)}{r(2-2\cosh\tau+\tau\sinh\tau)}  & \text{if } \sigma(z)<0,\\
\frac{4}{r} & \text{if } \sigma(z)=0,
\end{cases}
\]
where $\sigma(z)=r(z)^2((v_0r(z))^2+k)$ and $\tau(z)=\sqrt{|\sigma(z)|}$.
\end{cor}

\begin{proof}[Proof of Theorem \ref{HessExplicit}]
Let $\varphi_t(x)=\pi(e^{t\vec H}(d\ds _x))$. Assume that $z$ is a point where $\ds$ is twice differentiable. Let $\Lambda$ be the image of the linear map $d((d\ds)_z):T_zM\to T_{d\ds_z}T^*M$. Let $E_t=(e_1(t),e_2(t),e_3(t))^T,F_t=(f_1(t),f_2(t),f_3(t))^T$ be a Darboux frame at $df_z$ and let $\iota_i=d\pi(f_i(0))$. Let $A_t$ and $B_t$ be the matrices with $ij$-th entry $a_{ij}(t)$ and $b_{ij}(t)$, respectively, defined by
\[
d(d\ds_y)(\iota_i)=\sum_{j=1}^3\left(a_{ij}(t)e_j(t)+b_{ij}(t)f_j(t)\right).
\]

We define the matrix $S_t$ by $S_t=B_t^{-1}A_t$. Since $\pi(e^{1\cdot \vec H}(d\ds_x))=x_0$ for all $x$, we have $\lim_{t\to 1}S_t^{-1}=0$. The same argument as in Theorem \ref{volgrowth} shows that
\[
\dot S_t-R+S_tC_1+C_1^TS_t-S_tC_2S_t=0,
\]
where $R=\left(
\begin{array}{ccc}
	\sigma & 0 & 0\\
	0 & 0 & 0\\
	0 & 0 & 0
	\end{array}\right)$ and $\sigma$ is defined by
\[
\sigma=h_0(d\ds_z)^2+2H(d\ds_z)k=(v_0\ds(z))^2-2\ds(z)k=r(z)^2((v_0r(z))^2+k).
\]

By the proof of Proposition \ref{Hessian} and $B_0=I$, we have $\mathfrak H\ds (z)=-A_0=-S_0$. By the result in \cite{Le}, we can compute $S_t$ and it is given by
\[
S_t=
\begin{cases}
\left(
\begin{array}{ccc}
\frac{\tau_0(\sin\tau_t-\tau_t\cos\tau_t)}{2-2\cos\tau_t-\tau_t\sin\tau_t} & \frac{\tau_0^2(1-\cos\tau_t)}{2-2\cos\tau_t-\tau_t\sin\tau_t} & 0 \\
\frac{\tau_0^2(1-\cos\tau_t)}{2-2\cos\tau_t-\tau_t\sin\tau_t} & \frac{\tau_0^3\sin\tau_t}{2-2\cos\tau_t-\tau_t\sin\tau_t} & 0 \\
0 & 0 & \frac{1}{1-t}
\end{array}\right) & \text{if } \sigma>0,\\
\left(
\begin{array}{ccc}
\frac{\tau_0(\tau_t\cosh\tau_t-\sinh\tau_t)}{2-2\cosh\tau_t+\tau_t\sinh\tau_t} & \frac{\tau_0^2(\cosh\tau_t-1)}{2-2\cosh\tau_t+\tau_t\sinh\tau_t} & 0 \\
\frac{\tau_0^2(\cosh\tau_t-1)}{2-2\cosh\tau_t+\tau_t\sinh\tau_t} & \frac{\tau_0^3\sinh\tau_t}{2-2\cosh\tau_t+\tau_t\sinh\tau_t} & 0 \\
0 & 0 & \frac{1}{1-t}
\end{array}\right) & \text{if } \sigma<0,\\
\left(
\begin{array}{ccc}
\frac{4}{1-t} & \frac{6}{(1-t)^2} & 0 \\
\frac{6}{(1-t)^2} & \frac{12}{(1-t)^3} & 0 \\
0 & 0 & \frac{1}{1-t}
\end{array}\right) & \text{if } \sigma=0,
\end{cases}
\]
where $\tau_t=(1-t)\sqrt{|\sigma|}$.

By setting $t=0$, we obtain the result.
\end{proof}

\smallskip

\section{Subriemannian Hessian and Laplacian Comparison Theorem}\label{subLaplace}

In this section, we prove a Hessian and a Laplacian comparison theorem in our subriemannian setting. Let $(M^1,\Delta^1,g^1)$ and $(M^2,\Delta^2,g^2)$ be three dimensional contact subriemannian manifolds. Let $x_0^i$ be a point on the manifold $M^i$, let $r_i$ be the subriemannian distance from the point $x_0^i$, and let $\ds_i=-\frac{1}{2}r_i^2$. Let $v_0^i$ be the Reeb field in $M^i$ and let $(R_t^{11})^i$ and $(R_t^{22})^i$ be the curvature invariant on $M^i$ introduced in section \ref{subriemannian}. Finally let
\[
R_t^i=\left(
\begin{array}{ccc}
(R_t^{11})^i & 0 & 0 \\
0 & (R_t^{22})^i & 0\\
0 & 0 & 0
\end{array}\right).
\]

\begin{thm}\label{HessianCompareI}(Subriemannian Hessian Comparison Theorem I)
Let $z_1$ and $z_2$ be points on the three dimensional contact subriemannian manifolds $M^1$ and $M^2$, respectively, such that $\ds_i$ is twice differentiable at $z_i$. Assume that $R_t^1\Big|_{(d\ds_1)_{z_1}}\leq R_t^2\Big|_{(d\ds_2)_{z_2}}$ for all $t$ in the interval $[0,1]$. Then
\[
\mathfrak H \ds_1(z_1)\leq \mathfrak H \ds_2(z_2).
\]
\end{thm}

\begin{rem}
By the result in \cite{CaRi}, $\ds_i$ is twice differentiable Lebesgue almost everywhere.
\end{rem}

If we restrict to Sasakian manifolds, then we have the following.

\begin{thm}\label{HessianCompareII}(Subriemannian Hessian Comparison Theorem II)
Let $d$ be the subriemannian distance function of a Sasakian manifold $(M,\Delta,g)$ and let $\ds(x)=-\frac{1}{2}d^2(x,x_0)$, where $x_0$ is a point on $M$. Assume that the Tanaka-Webster curvature $\kappa$ of $M$ satisfies $\kappa\geq k$ (resp. $\kappa\leq k$). Then the following holds wherever $f$ is twice differentiable.
\[
\mathfrak H\ds(z)\geq
\begin{cases}
-\left(
\begin{array}{ccc}
\frac{\tau(\sin\tau-\tau\cos\tau)}{2-2\cos\tau-\tau\sin\tau} & \frac{\tau^2(1-\cos\tau)}{2-2\cos\tau-\tau\sin\tau} & 0 \\
\frac{\tau^2(1-\cos\tau)}{2-2\cos\tau-\tau\sin\tau} & \frac{\tau^3\sin\tau}{2-2\cos\tau-\tau\sin\tau} & 0 \\
0 & 0 & 1
\end{array}\right) & \text{if } \sigma(z)>0,\\
-\left(
\begin{array}{ccc}
\frac{\tau(\tau\cosh\tau-\sinh\tau)}{2-2\cosh\tau+\tau\sinh\tau} & \frac{\tau^2(\cosh\tau-1)}{2-2\cosh\tau+\tau\sinh\tau} & 0 \\
\frac{\tau^2(\cosh\tau-1)}{2-2\cosh\tau+\tau\sinh\tau} & \frac{\tau^3\sinh\tau}{2-2\cosh\tau+\tau\sinh\tau} & 0 \\
0 & 0 & 1
\end{array}\right) & \text{if } \sigma(z)<0,\\
-\left(
\begin{array}{ccc}
4 & 6 & 0 \\
6 & 12 & 0 \\
0 & 0 & 1
\end{array}\right) & \text{if } \sigma(z)=0,
\end{cases}
\]
(resp. $\leq$) where $\sigma=(v_0\ds)^2-2k\ds$ and $\tau=\sqrt{|\sigma|}$.
\end{thm}

If we combine Theorem \ref{HessianCompareII} and Proposition \ref{HessLaplaceRelate}, then we have the following sub-Laplacian comparison theorem.

\begin{cor}(Sub-Laplacian Comparison Theorem) \label{LaplacianCompare}
Under the notation and assumptions of Theorem \ref{HessianCompareII}, the following holds wherever $r$ is twice differentiable.
\[
\Delta_H r(z)\leq
\begin{cases}
\frac{\tau(\sin\tau-\tau\cos\tau)}{r(2-2\cos\tau-\tau\sin\tau)} & \text{if } \sigma(z)>0,\\
\frac{\tau(\tau\cosh\tau-\sinh\tau)}{r(2-2\cosh\tau+\tau\sinh\tau)}  & \text{if } \sigma(z)<0,\\
\frac{4}{r} & \text{if } \sigma(z)=0,
\end{cases} \quad (\text{resp. }\geq)
\]
where $r(x)=d(x,x_0)$, $\sigma(z)=r(z)^2((v_0r(z))^2+k)$, and $\tau(z)=\sqrt{|\sigma(z)|}$.
\end{cor}

\begin{proof}[Proof of Theorem \ref{HessianCompareI}]
Let $\varphi^i_t(x)=\pi(e^{t\vec H}((d\ds_i)_x))$ and let
\[
E_t^i=(e_1^i(t),e_2^i(t),e_3^i(t))^T, \quad F_t^i=(f_1^i(t),f_2^i(t),f_3^i(t))^T
\]
be a Darboux frame at $(d\ds_i)_{z_i}$ and let $\iota^i_j=d\pi(f_j(0))$. Let $A_t^i$ and $B_t^i$ be the matrices with $jk$-th entry $a_{jk}^i(t)$ and $b_{jk}^i(t)$, respectively, defind by
\[
d(d\ds_{z_i})(\iota_j)=\sum_{k=1}^3(a_{jk}^i(t)e_k^i(t)+b_{jk}^i(t)f_k^i(t)).
\]

We define the matrix $S_t^i$ by $(B_t^i)^{-1}A_t^i$. As in the proof of Theorem \ref{HessExplicit}, we have $\mathfrak H\ds_i(z_i)=-S_0^i=-A_0^i$ and
\[
\dot S_t^i-R_t^i+S_t^iC_1+C_1^TS_t^i-S_t^iC_2S_t^i=0\quad \lim_{t\to 1}(S_t^i)^{-1}=0,
\]
where $R_t^i$ here denotes $R_t^i\Big|_{df_{z_i}}$.

Therefore, by assumption and the result in \cite{Ro}, we have the following as claimed
\[
\mathfrak H\ds_2=-S^2_t\geq -S^1_t=\mathfrak H\ds_1.
\]
\end{proof}

\begin{proof}[Proof of Theorem \ref{HessianCompareII}]
Let us first assume that $\kappa\geq k$. Here we use the same notation as in proof of Theorem \ref{HessExplicit}. The matrix $S_t$ satisfies the equation
\[
\dot S_t-R_t+S_tC_1+C_1^TS_t-S_tC_2S_t=0, \quad  \lim_{t\to 1}(S_t)^{-1}=0,
\]
where
\[
R_t=\left(\begin{array}{ccc}
\sigma & 0 & 0\\
0 & 0 &0\\
0 & 0 & 0
\end{array}\right).
\]

By the result in \cite{Ro}, we have $S_t\leq S_t^k$, where $S^k_t$ is the solution of
\[
\dot S_t-R^k+S_tC_1+C_1^TS_t-S_tC_2S_t=0, \quad  \lim_{t\to 1}(S_t)^{-1}=0,
\]
and
\[
R^k=\left(\begin{array}{ccc}
(v_0\ds)^2-2k\ds & 0 & 0\\
0 & 0 &0\\
0 & 0 & 0
\end{array}\right).
\]

If we set $t=0$, then we get $\mathfrak H\ds=-S_0\geq -S_0^k$. Finally the matrix $S_0^k$ can be computed using the result in \cite{Le} which gives the first statement of the theorem. The reverse inequalities under the assumption $\kappa\leq k$ are proved in a similar way.
\end{proof}

\smallskip

\section{Cheeger-Yau Type Theorem in Subriemannian Geometry}\label{ChYau}

In this section, we give a lower bound on the solution of the subriemannian heat equation $\dot u=\Delta_Hu$ in the spirit of its Riemannian analogue in \cite{ChYa}. More precisely, let $\phi$ be the function defined by
\[
\phi(s)=\begin{cases}
\frac{\sqrt k(\sin(s\sqrt k)-s\sqrt k\cos(s\sqrt k))}{2-2\cos(s\sqrt k)-s\sqrt k\sin(s\sqrt k)} & \text{if } k>0\\
\frac{4}{s} & \text{if } k=0.
\end{cases}
\]

Let $h=h(t,s):[0,\infty)\times (0,\infty)\to\Real$ be a smooth solution to the following equation
\begin{equation}\label{geqn1}
\dot h=h''+h'\phi.
\end{equation}
where $\dot h$ and $h'$ denotes the derivative with respect to $t$ and $s$, respectively.

\begin{thm}
Let $(M,\Delta,g)$ be a three dimension Sasakian manifold with non-negative Tanaka-Webster curvature. Let $h$ be a solution of (\ref{geqn1}) which satisfies the conditions
\[
h'(0,s)\leq 0,\quad \lim_{s\to 0} h'(t,s)\leq 0.
\]
Let $x_0$ be a point on the manifold $M$ and let $r(\cdot)=d(x_0,\cdot)$, where $d$ is the subriemannian distance function. Let $\Omega$ be an open set which contains $x_0$ and have smooth boundary $\partial \Omega$. Let $u=u(t,x)$ be a smooth solution to the subriemannian heat equation $\dot u=\Delta_H u$ on $[0,\infty)\times M\backslash\,\Omega$ which satisfies
\[
u(0,\cdot)\geq h(0,r(\cdot)),\quad u(t,y)\geq h(t,r(y)) \quad y\in\partial\Omega.
\]
Then we have $u\geq h\circ r$.
\end{thm}

\begin{rem}
When $k=0$, the function
\[
h(t,s)=(t+\epsilon)^{-5/2}e^{\frac{-s^2}{4(t+\epsilon)}}
\]
is a solution to the equation (\ref{geqn1}) for every $\epsilon>0$.
\end{rem}

\begin{proof}
Let $r$ be the subriemannian distance function from the point $x_0$ (i.e. $r(x)=d(x_0,x)$). By Corollary \ref{LaplacianCompare} and the chain rule, the following holds $\eta$-a.e.
\[
\Delta_H r\leq
\begin{cases}
\frac{\tau(\sin\tau-\tau\cos\tau)}{r(2-2\cos\tau-\tau\sin\tau)} & \text{if } \sigma>0\\
\frac{4}{r} & \text{if } \sigma=0
\end{cases}
\]
where $\sigma=r^2((v_0r)^2+k)$ and $\tau=r\sqrt{(v_0r)^2+k}$.

If we differentiate (\ref{geqn1}) with respect to $s$. Then we get
\begin{equation}\label{geqn2}
\dot h'=h'''+h''\phi +h'\phi'.
\end{equation}

By the maximum principle, we see that $h'(t,s)\leq 0$ for all $t$ and for all $s$ since $h'(0,s)\leq 0$ for all $s$ and $h'(t,0)\leq 0$ for all $t$ by assumptions.

Therefore, the following holds wherever $r$ is twice differentiable.
\begin{equation}\label{geqn3}
\begin{split}
\Delta_H(h(t,r))&=h''(t,r)+h'(t,r)\Delta_Hr\\
&\geq h''(t,r)+h'(t,r)\phi(r)\\
&= \dot h(t,r).
\end{split}
\end{equation}

Let $(t_0,z)$ be a local minimum of the function $G(t,x)=u(t,x)-h(t,r(x))+\delta t$, where $\delta$ is a positive constant. Let us assume that $t_0>0$ and $z$ in contained in the interior of $M\backslash\,\Omega$. By the result in \cite{CaRi}, $r$ is locally semiconcave on $M\backslash\,\{x_0\}$. Since $g$ is nondecreasing in $s$, $G(t,x)$ is locally semi-concave on $M\backslash\,\{x_0\}$ as well. Therefore, by \cite[Theorem 2.3.2]{Ca}, we can find a sequence of points $z_i$ on the manifold $M$ converging to $z$ and a sequence of numbers $\epsilon_i$ converging to $0$ such that
\[
\Delta_HG(t_0,z_i)\geq-\epsilon_i.
\]

Since $u$ is the solution of the subriemannian heat equation, it follows from (\ref{geqn3}) that
$\frac{d}{dt}G(t_0,z_i)\geq -\epsilon_i +\delta$. If we let $i$ go to $\infty$, then we have $0=\frac{d}{dt}G(t_0,z)\geq \delta$ which is a contradiction.

Since we have the condition $u(t,x)\geq h(t,r(x))$ for all points $x$ on the boundary of $\Omega$ and $u(0,\cdot)\geq h(0,r(\cdot))$, it follows that $G\geq 0$. Therefore, if we let $\delta$ go to $0$, then we have $u\geq h\circ r$ as claimed.
\end{proof}

\end{document}